\documentclass[11pt, a4paper]{amsart}

\usepackage{amsfonts,amsmath,amssymb, amscd,fullpage}
\usepackage[all]{xy}

\newtheorem{theorem}{Theorem}[section]
\newtheorem{lemma}[theorem]{Lemma}

\newtheorem{corollary}[theorem]{Corollary}
\newtheorem{remark}[theorem]{Remark}

\newcommand\pf{\begin{proof}}
\newcommand\epf{\end{proof}}

\newcommand\co{\operatorname{co}}

\newcommand\Alg{\operatorname{Alg}}

\DeclareMathOperator{\id}{id}

\DeclareMathOperator{\Gal}{Gal}

\numberwithin{equation}{section}

\title{The group of bi-Galois objects over the coordinate algebra of  the Frobenius-Lusztig kernel of ${\rm SL}(2)$}

\author{Julien Bichon}
\address{
Laboratoire de Math\'ematiques,
Universit\'e Blaise Pascal,
Complexe universitaire des C\'ezeaux,
63171~Aubi\`ere Cedex, France}
\email{Julien.Bichon@math.univ-bpclermont.fr}


\begin{document}

\begin{abstract}
 We construct, for $q$ a root of unity of odd order, an embedding of the projective special linear group ${\rm PSL}(n)$ into the group of bi-Galois objects over $u_q(sl(n))^*$,  the coordinate algebra of  the Frobenius-Lusztig kernel of ${\rm SL}(n)$, which is shown to be an isomorphism at $n=2$. 
\end{abstract}

\maketitle

\section{introduction}

Let $\mathcal C$ be finite tensor category over $k$, an algebraically closed field of characteristic zero. The Brauer-Picard group of $\mathcal C$ \cite{eno,dn}, denoted ${\rm BrPic}(\mathcal C)$, consists of equivalence classes of invertible exact $\mathcal C$-bimodule categories, endowed with the group law induced by a natural (but technically involved) notion of tensor product. 
It is an important invariant of $\mathcal C$, but it seems to be very difficult to compute, or even difficult to find non-trivial subgroups of it. The Brauer-Picard group of $\mathcal C$ also coincides with the Brauer group of
the Drinfeld center of $\mathcal C$ as defined in \cite{vOZ}, in terms of Azumaya algebras, see \cite{dn}.
If $\mathcal C= {\rm mod}(H)$ is the category of finite-dimensional representations of a finite-dimensional Hopf algebra $H$, then the Brauer-Picard group is known as the strong Brauer of $H$ \cite{coz},
which is also notoriously known to be difficult to compute: see \cite{cc} for the latest developments (for the case of the Sweedler algebra) before the new technology from \cite{eno,dn}, and see \cite{boni,mom} for recent computations. 

A key result to understand the Brauer-Picard group is given in  \cite{eno,dn},
 where it is shown to be isomorphic with ${\rm Aut}^{\rm br}(\mathcal Z(\mathcal C))$, the group of isomorphism classes of braided auto-equivalences of the Drinfeld center of $\mathcal C$, thus providing a more tractable description.
Since any tensor auto-equivalence of $\mathcal C$ obviously induces a braided tensor auto-equivalence of $\mathcal Z(\mathcal C)$, there exists a group morphism from  ${\rm Aut}^{\otimes}(\mathcal C)$ (the group of isomorphism classes of tensor auto-equivalences of $\mathcal C$) into
${\rm Aut}^{\rm br}(\mathcal Z(\mathcal C))$, and hence into ${\rm BrPic}(\mathcal C)$.
Therefore the study of  ${\rm Aut}^{\otimes}(\mathcal C)$ seems to be an important step to understand the structure of the Brauer-Picard group.

When $\mathcal C= {\rm mod}(H)$ is the category of finite-dimensional representations of a finite-dimensional Hopf algebra $H$, it is shown by Schauenburg in \cite{sc1} that  the group  ${\rm Aut}^{\otimes}({\rm mod}(H))$ is isomorphic to ${\rm BiGal}(H^*)$, the group of $H^*$-bi-Galois objects. This result is very useful to construct tensor auto-equivalences, and Schauenburg has moreover developed some powerful techniques to describe and study (bi-)Galois objects over a given Hopf algebra \cite{scha,sc00}. 
The aim of this paper is to show how to use these techniques to compute the group of bi-Galois objects over $u_q(sl(2))^*$ at an odd root of unity, the coordinate algebra of  the Frobenius-Lusztig kernel of ${\rm SL}(2)$, and hence the group ${\rm Aut}^{\otimes}({\rm mod}(u_q(sl(2)))$, as follows.

\begin{theorem}\label{main}
 Let $q$ be a root of unity of odd order $N>1$, and let $n\geq 2$. There exists an injective group morphism
$${\rm PSL}(n) \longrightarrow {\rm BiGal}(u_q(sl(n))^*)$$
which is an isomorphism at $n=2$.
\end{theorem}

An important difference between the present case of  $u_q(sl(2))^*$ and the few other known computations of groups of bi-Galois objects for finite-dimensional non-cosemisimple Hopf algebras \cite{sc00, bi,memo} is that  $u_q(sl(2))^*$ is not pointed.  Thus its seems difficult to use the techniques of these papers. 
The key tool used here will be a result by Schauenburg in \cite{scha} (Corollary 3.3 there), yielding information on Galois objects over Hopf algebras involved in an exact sequence,  together with the fact that the cleft Galois objects over $\mathcal O({\rm SL}_q(2))$ are trivial if $q \not=1$ (\cite{au}).

As a consequence of Theorem \ref{main}, we also show that the lazy cohomology \cite{bc} of $u_q(sl(2))^*$ is trivial.

Note that for $q$ not a root of unity, the group of tensor auto-equivalences of the (semisimple) tensor category of finite-dimensional $U_q(sl(n))$-modules (or more generally of $U_q({\mathfrak g})$-modules for $\mathfrak g$ a semisimple Lie algebra) is described in \cite{nt}, and is much smaller than the one we get at roots of unity for $u_q(sl(n))$.

The paper is organized as follows. Section 2 consists of preliminaries, mainly  on Galois and bi-Galois objects over Hopf algebras. In particular we provide a detailed exposition of (a particular case of) Schauenburg's Corollary 3.3 in \cite{scha}, of which we provide a variation adapted to bi-Galois objects (Corollary \ref{corobigal}).
The proof of  Theorem \ref{main} is given in Section 3, that the reader might consult first to see how the group morphism in the statement is constructed.

\medskip

\noindent
\textbf{Acknowledgments.} I wish to thank Dmitri Nikshych for useful discussions. 

\section{Preliminaries}

\subsection{Notation and conventions} We assume that the reader is familiar with the theory of Hopf algebras, as in \cite{mon}. In particular we use Sweedler notation for coproducts and  comodules in the standard way.
A pointed set is a set with a distinguished element. A sequence of maps of pointed sets $(X,*_X) \rightarrow (Y,*_Y) \rightarrow (Z,*_Z)$ is said to be exact if the fiber of $*_Z$ is exactly the image of $X$.

\subsection{Galois and bi-Galois objects}
Let $H$ be a Hopf algebra.
A right $H$-Galois object 
is a non-zero right $H$-comodule algebra $T$ such that the linear map
$\kappa_r$ (the ``canonical'' map) defined by 
\begin{align*}
\kappa_r : T \otimes T & \longrightarrow T \otimes H \\
t \otimes t' & \longrightarrow tt'_{(0)} \otimes t'_{(1)}
\end{align*}
 is bijective.
An $H$-Galois object morphism is an $H$-colinear algebra 
morphism. 
The set of isomorphism classes of $H$-Galois objects
is denoted  Gal$(H)$. 

We now list a number of useful properties and constructions, that will be used freely later.

$\bullet$ Any morphism of  $H$-Galois objects  is an isomorphism (\cite{schnei}, Remark 3.11).

$\bullet$ An $H$-Galois object $T$ is trivial (i.e. $T \simeq H$ as $H$-comodule algebras) if and only if there exists an algebra map $\phi : T \rightarrow k$, an $H$-Galois object isomorphism $T \simeq H$ being then of the form $t \mapsto \phi(t_{(0)})t_{(1)}$. The set ${\rm Gal}(H)$ is viewed as a pointed set, the distinguished element being the isomorphism class of the trivial Galois object. 

$\bullet$ An $H$-Galois object $T$ is said to be cleft if $T\simeq H$ as right $H$-comodules. We denote by ${\rm Cleft}(H)$ the subset of $ {\rm Gal}(H)$ consisting of isomorphism classes of cleft $H$-Galois objects. If $H$ is finite-dimensional then ${\rm Cleft}(H)={\rm Gal}(H)$ (\cite{kc}), but this is not true in general, see \cite{bi1}.

$\bullet$  Let $T$ be an $H$-Galois object, and consider the map 
\begin{align*}
 H & \longrightarrow T \otimes T \\
h & \longmapsto \kappa_r^{-1}(1 \otimes h) =: h^{[1]} \otimes h^{[2]}
\end{align*}
It endows $T$ with a right $H$-module structure given by $t\cdot h$=   $h^{[1]} t h^{[2]}$, called the Miyashita- Ulbrich action, see e.g. \cite{dt}. A morphism of $H$-Galois objects  $T \rightarrow Z$ automatically commutes with the Miyashita-Ulbrich actions.

$\bullet$ A Hopf algebra map $\pi : H \rightarrow L$ induces a map ${\rm Gal}(L) \rightarrow {\rm Gal}(H)$, sending (the isomorphism class of) an $L$-Galois object $T$ to (the isomorphism class of) the $H$-Galois object $T\square_L H$, where $H$ has the left $L$-comodule structure induced by $\pi$ and the right $H$-comodule structure is induced by the coproduct of $H$. See e.g. \cite{schnei}, Remark 3.11. The $H$-Galois object  $T\square_L H$ is cleft if $T$ is $L$-cleft.
The Miyashita-Ulbrich action on $T\square_L H$ is given by 
$$(\sum_i t_i \otimes h_i)\cdot h = \sum_i \pi(h_{(2)})^{[1]} t_i \pi(h_{(2)})^{[2]} \otimes S(h_{(1)})h_ih_{(3)}$$

A left $H$-Galois object 
is a non-zero left $H$-comodule algebra $T$ such that the linear map
$\kappa_l$ defined by 
\begin{align*}
\kappa_l : T \otimes T & \longrightarrow H \otimes T \\
t \otimes t' & \longrightarrow t_{(-1)} \otimes t_{(0)} t'
\end{align*}
 is bijective. The previous considerations have  adaptations to left Galois objects. 

Let $H$ and $L$ be Hopf algebras. An  
 $L$-$H$-bi-Galois object \cite{sc1} is an $L$-$H$-bicomodule algebra $T$ which is both
a left $L$-Galois object and a right $H$-Galois object.
A morphism of $L$-$H$-bi-Galois object is a bicolinear algebra map, and the set of isomorphism classes of 
 $L$-$H$-bi-Galois objects is denoted ${\rm BiGal}(L,H)$, with ${\rm BiGal}(H)= {\rm BiGal}(H,H)$ (an $H$-$H$-bi-Galois object is simply called an $H$-bi-Galois object).
Here is a list of useful facts and constructions regarding bi-Galois objects, to be used freely in the rest of the paper.

$\bullet$  An $H$-bi-Galois object $T$ is trivial (i.e. $T \simeq H$ as $H$-bicomodule algebras) if and only if there exists an algebra map $\phi : T \rightarrow k$ satisfying $\phi(t_{(0)})t_{(1)}= \phi(t_{(0)})t_{-1}$ for any $t \in T$, an $H$-bi-Galois object isomorphism $T \simeq H$ being then of the form $t \mapsto \phi(t_{(0)})t_{(1)}$.

$\bullet$ Any $L$-$H$-bi-Galois object $T$ induces an equivalence of $k$-linear tensor categories between the categories of right comodules over $L$ and $H$,
 given by $V \mapsto V\square_L T$, and conversely
any such equivalence arises in this way, See \cite{sc1}, Corollary 5.7.

$\bullet$ If $T$, $Z$ are $H$-bi-Galois objects, so is the cotensor product $T\square_H Z$, and this construction induces a group structure on ${\rm BiGal}(H)$. With the construction of the previous item, this defines a group morphism from ${\rm BiGal}(H)$ to ${\rm Aut}^\otimes({\rm Comod}(H))$, the group of isomorphism classes of $k$-linear tensor auto-equivalences of ${\rm Comod}(H)$, which is an isomorphism. Again see  \cite{sc1}, Corollary 5.7.

$\bullet$ An $H$-bi-Galois object $T$ is said to be bicleft if $T\simeq H$ as bicomodules. The isomorphism classes of bicleft bi-Galois objects form a normal subgroup of ${\rm BiGal}(H)$, denoted ${\rm Bicleft}(H)$, isomorphic to the lazy cohomology group ${\rm H}^2_{\ell}(H)$ studied in \cite{bc}. The elements in ${\rm Bicleft}(A)$ correspond to isomorphism classes of tensor auto-equivalences that are isomorphic to the identity functor as linear functors. 
The group ${\rm H}^2_{\ell}(H)$ can be described by using lazy cocycles: cocycles (see \cite{mon}) $\sigma : H \otimes H \rightarrow k$ satisfying $\sigma(x_{(1)},y_{(1)})x_{(2)}y_{(2)}= \sigma(x_{(2)},y_{(2)})x_{(1)}y_{(1)}$
for any $x,y \in H$.
See \cite{bc}.

$\bullet$ Let $T$ be an $H$-bi-Galois object and let $f : H \rightarrow H$ be a Hopf algebra automorphism. 
From this one defines a new $H$ bi-Galois object ${^f\! T}$ which is $T$ as a right Galois object, and whose left comodule structure is obtained by composing the left coaction of $T$ with $f$ in the obvious way. 
This defines an action of ${\rm Aut}_{\rm Hopf}(H)$ on ${\rm Bigal}(H)$.  The $H$-bi-Galois object ${^f\! T}$ is isomorphic with $T$ if and only $f$ is co-inner, i.e. there exists $\phi \in {\rm Alg}(H,k)$ such that $f = \phi*{\rm id}*\phi^{-1}$.

$\bullet$ If $T$ and $Z$ are $H$-bi-Galois objects that are isomorphic as right $H$-comodule algebras, then there exists a Hopf algebra automorphism $f : H \rightarrow H$ such that $Z\simeq {^f\! T}$ as bi-Galois objects. See Lemma 3.11 in \cite{sc1}.

\subsection{Exact sequences of Hopf algebras} Recall that a sequence  of Hopf algebra maps
\begin{equation*}k \to A \overset{i}\to H \overset{\pi}\to L \to k\end{equation*} is said to be exact  if the following
conditions hold:
\begin{enumerate}\item $i$ is injective, $\pi$ is surjective and $\pi i$ = $\varepsilon 1$,
\item ${\rm Ker}(\pi) =HA^+ =A^+H$, where $A^+=A\cap{\rm Ker}(\varepsilon)$,
\item $A = H^{\co \pi} = \{ x\in H:\, (\id \otimes \pi)\Delta(x) = x \otimes 1
\} = {^{\co \pi}H} = \{ x \in H:\, (\pi \otimes \id)\Delta(x) = 1 \otimes x
\}$. \end{enumerate}
A sequence as above satisfying (1), (2) and with $H$ faithfully flat as a right or left $A$-module is automatically exact (see e.g. \cite{ad}, Proposition 1.2.4, or \cite{schn}, Lemma 1.3).

In an exact sequence as above, the Hopf subalgebra $A \subset H$ is automatically normal, i.e. for $a \in A$ and $h \in H$, we have 
$$h_{(1)} a S(h_{(2)}) \in A, \ S(h_{(1)})ah_{(2)} \in A$$
See e.g. Lemma 3.4.2 in \cite{mon}. We will denote by ${\rm Alg}_H(A,k)$ the subset of ${\rm Alg}(A,k)$ consisting of algebra maps 
$\varphi : A \rightarrow k$ satisfying $\varphi(S(h_1)ah_2) = \varepsilon(h)\varphi(a)$, for any $h \in H$ and $a \in A$.
The subset ${\rm Alg}_H(A,k)$ is a subgroup of $\Alg(H,k)$: it is clearly stable under the convolution product, and stable under inverses by the following computation ($a\in A$, $h \in H$):
\begin{align*}
 \varphi S(S(h_{(1)})ah_{(2)}) &=  \varepsilon(a_{(1)}) \varphi S(S(h_{(1)})a_{(2)}h_{(2)}) = 
\varphi S(a_{(1)}) \varphi(a_{(2)}) \varphi S(S(h_{(1)})a_{(3)}h_{(2)})\\
& = \varphi S(a_{(1)}) \varphi(S(h_{(2)})a_{(2)}h_{(3)}) \varphi S(S(h_{(1)})a_{(3)}h_{(4)}) \\
& = \varphi S(a_{(1)}) \varphi((S(h_{(1)})a_{(2)}h_{(2)})_{(1)}) \varphi S((S(h_{(1)})a_{(2)}h_{(2)})_{(2)})\\
& = \varphi S(a_{(1)}) \varepsilon(S(h_{(1)})a_{(2)}h_{(2)})=\varphi S(a)\varepsilon(h)
\end{align*}

\subsection{Schauenburg's exact sequence}
We now recall our main tool to prove Theorem \ref{main}.
\begin{theorem}[\cite{scha}, Corollary 3.3]\label{schexact}
Let \begin{equation*}k \to A \overset{i}\to H \overset{\pi}\to L \to
k\end{equation*} be an exact sequence of Hopf algebras with $H$ faithfully flat as a right $A$-module. Then we have an exact sequence of pointed sets
 $$1 \longrightarrow \Alg(L,k) \longrightarrow \Alg(H,k) \longrightarrow \Alg_H(A,k) \longrightarrow \Gal(L) \longrightarrow \Gal(H)$$
\end{theorem}
The result above is in fact a special case of what is proved in \cite{scha}. We include, for the sake of completeness, a sketch of proof (the arguments being simpler in our particular setting). 

The first two maps on the left are those induced by the given Hopf algebra maps in the obvious way, while the one on the right is also the obvious one. The key point is to describe the map $\Alg_H(A,k) \longrightarrow \Gal(L)$, a generalisation of the transgression map in group cohomology.

We fix an exact sequence as in the statement of the theorem.
For $\varphi \in \Alg_H(A,k)$, we define 
$\varphi' : A \longrightarrow H$ by 
$\varphi'(a)= \varphi(a_{(1)}) a_{(2)}$. The map $\varphi'$ is an algebra map and satisfies $ \varphi'(S(h_{(1)})a h_{(2)})=S(h_{(1)})\varphi'(a)h_{(2)}$ for any $a \in A$, $h \in H$.

\begin{lemma} \label{traright} Let $\varphi \in \Alg_H(A,k)$.
The left ideal $H\varphi'(A^+)$ is a two-sided ideal in $H$. The map
 \begin{align*}
  \rho : H / H\varphi'(A^+) &\longrightarrow H/H\varphi'(A^+) \otimes L \\
  \overline{h} &\longmapsto \overline{h_{(1)}} \otimes \pi(h_{(2)})
 \end{align*}
defines a right $L$-comodule algebra structure on $H/ H\varphi'(A^+)$ making it into a right $L$-Galois object.
\end{lemma}

\begin{proof}
We have
$$\varphi'(a)h = h_{(1)}S(h_{(2)})\varphi'(a)h_{(3)} = h_{(1)} \varphi'(S(h_{(2)})ah_{(3)}))$$
for any $a \in A$, $h \in H$
and hence $\varphi'(A^+)H \subset H \varphi'(A^+)$. This shows that $H\varphi'(A^+)$ is a two-sided ideal.
We have $\varphi'(A^+) \subset {\rm Ker}(\varphi \circ S)$ since for $a \in A$ we have $\varphi\circ S(\varphi'(a)) = \varepsilon(a)$, hence $\varphi'(A^+)$ is a proper ideal in $A$. A standard argument using the faithful flatness of $H$ as a right $A$-module then shows that $H\varphi'(A^+)$ is also a proper ideal of $H$, so that $H/H\varphi'(A^+)$ is a non-zero algebra. Let
\begin{align*}
  \rho_0 : H  &\longrightarrow H/H\varphi'(A^+) \otimes L \\
  h &\longmapsto \overline{h_{(1)}} \otimes \pi(h_{(2)})
 \end{align*}
This is an algebra map and for $a \in A^+$, we have
\begin{align*}
 \rho_0(\varphi'(a)) & = \overline{\varphi'(a)_{(1)}} \otimes \pi(\varphi'(a)_{(2)}) 
  =  \overline{\varphi'(a_{(1)})} \otimes \pi(a_{(2)}) 
 =  \overline{\varepsilon(a_{(1)})} \otimes \pi(a_{(2)}) = \overline{1} \otimes \pi(a) =0 
\end{align*}
since for $a \in A$, $\overline{\varphi'(a)}= \varepsilon(a)\overline{1}$. This shows that
$\rho_0$ induces the announced algebra map, which is clearly co-associative, and we  get our comodule algebra structure. 

For $a \in A$, we have
$S(a) = \varphi'(\varphi(a_{(2)})S(a_{(1)}))$, and hence 
$\overline{S(a)}=\varphi(a)\overline{1}$.  We deduce that
$$\overline{S(a_{(1)})} \otimes \overline{a_{(2)}} =  \overline{1} \otimes \overline{\varphi(a_1)a_2} =
\varepsilon(a) \overline{1} \otimes \overline{1} $$
This identity shows that we have a map
\begin{align*}
L &\longrightarrow H/H\varphi'(A^+) \otimes H/H\varphi'(A^+) \\
\pi(h) &\longmapsto \overline{S(h_{(1)})} \otimes \overline{h_{(2)}}
\end{align*}
that enables us to construct an inverse for the canonical map in a standard manner, and we are done.
\end{proof}

\begin{proof}[Proof of Theorem \ref{schexact} (sketch)]
 We first examine  the exactness at $\Alg_H(A,k)$.
Let $\varphi \in \Alg(H,k)$. 
 For $a \in A$ we have $\varphi\circ S(\varphi'(a)) = \varepsilon(a)$.
 This shows that $\varphi \circ S : H \longrightarrow k$ induces an algebra map
 $H/ H\varphi'(A^+) \longrightarrow k$ and hence the $L$-Galois object
 $H/ H\varphi'(A^+)$ is trivial.
 
 Conversely, let $\varphi \in \Alg_H(A,k)$ be such that $H/ H\varphi'(A^+)$ is trivial: there exists an algebra map
 $\psi_0 : H/ H\varphi'(A^+) \longrightarrow k$. Define $\psi : H \longrightarrow k$ by $\psi(h) = \psi_0(\overline{S(h)})$. For $a\in A$, we have $\psi(a) = \psi_0(\overline{S(a)})= \psi_0(\varphi(a)\overline{1})
 =\varphi(a)$, and hence one can extend $\varphi$ to $H$.

We now check exactness at ${\rm Gal}(L)$. Let $\varphi \in \Alg_H(A,k)$. The map 
 \begin{align*}
  H &\longrightarrow H/ H\varphi'(A^+) \square_LH \\
  h &\longmapsto \overline{h_{(1)}} \otimes h_{(2)}
 \end{align*}
is an $H$-colinear algebra map between two $H$-Galois objects, and hence must be an isomorphism.

Conversely, Let $Z$ be right $L$-Galois object such that $Z \square_LH$ is trivial as right $H$-Galois object. Fix an $H$-comodule algebra isomorphism
\begin{align*}
\Phi : H &\longrightarrow Z\square_L H \\
 h &\longmapsto h^{(0)} \otimes h^{(1)}
\end{align*}
 The inverse has the form
 \begin{align*}
  Z \square_L H &\longrightarrow H \\
  z \otimes h &\longmapsto \psi(z \otimes h_{(1)})h_{(2)}
 \end{align*}
for an algebra map $\psi :  Z \square_L H \longrightarrow k$.
For $a \in A$, we have $1 \otimes a \in Z \square_L H$, so we define an algebra map
$\varphi : A \longrightarrow k$ by letting $\varphi(a) = \psi(1 \otimes a)$.
That $\varphi$ satisfies $\varphi(S(h_{(1)})ah_{(2)})=\varepsilon(h)\varphi(a)$ for $h \in H$ and $a \in A$ follows from the fact that an isomorphism between Galois objects automatically commutes with the Miyashita-Ulbrich actions (see  Subsection 2.2 for the Miyashita-Ulbrich action on $T\square_LH$). Consider now the algebra map
\begin{align*}
 f : H &\longrightarrow Z \\
 h &\longmapsto \varepsilon(h^{(1)}) h^{(0)}
\end{align*}
For $a \in A$ we have 
$$1 \otimes a = \psi(1\otimes a_{(1)}) (a_{(2)})^{(0)} \otimes (a_{(2)})^{(1)}$$ 
and hence $\varepsilon(a) = \psi(1\otimes a_{(1)}) (a_{(2)})^{(0)} \varepsilon ((a_{(2)})^{(1)})$. We get that
$$f(\varphi'(a)) = f(\psi(1\otimes a_{(1)})a_{(2)}) = \psi(1\otimes a_{(1)})(a_{(2)})^{(0)}\varepsilon((a_{(2)})^{(1)})=\varepsilon(a)1$$
and hence $f$ induces an algebra map $H/H\varphi'(A^+) \longrightarrow Z$. One checks that this is $L$-colinear
by using that $\Phi$ is $H$-colinear and has its values in $Z \square_LH$:
$$h^{(0)} \otimes (h^{(1)})_{(1)} \otimes (h^{(1)})_{(2)} = (h_{(1)})^{(0)} \otimes (h_{(1)})^{(1)} \otimes h_{(2)}$$
$$(h^{(0)})_{(0)} \otimes (h_{(0)})_{(1)} \otimes h^{(1)} = h^{(0)} \otimes \pi((h^{(1)})_{(1)}) \otimes (h^{(1)})_{(2)}$$  
We conclude that we have an isomorphism.
\end{proof}

\begin{remark}
{\rm  G\"unther \cite{gun} has provided a generalization of Schauenburg's exact sequence.}
\end{remark}

We now apply Schauenburg's exact sequence to bi-Galois objects. The basic observation is as follows.

\begin{lemma}\label{traleft}
 Let $\varphi \in \Alg_H(A,k)$. The map
 \begin{align*}
  \lambda : H / H\varphi'(A^+) &\longrightarrow L \otimes H/H\varphi'(A^+)  \\
  \overline{h} &\longmapsto \pi(h_{(1)}) \otimes \overline{h_{(2)}}
 \end{align*}
defines a left $L$-comodule algebra structure on $H/ H\varphi'(A^+)$ making it into a left $L$-Galois object, and hence an $L$-bi-Galois object for the right $L$-Galois structure in Lemma \ref{traright}.
\end{lemma}

\begin{proof}
We begin with the algebra map 
\begin{align*}
  \lambda_0 : H / H\varphi'(A^+) &\longrightarrow L \otimes H/H\varphi'(A^+)  \\
  h &\longmapsto \pi(h_{(1)}) \otimes \overline{h_{(2)}}\end{align*}
Similarly to the proof of lemma \ref{traright}, it induces the announced algebra map, which is clearly co-associative, and turns $H/H\varphi'(A^+)$ into a left $L$-comodule algebra, and into an $L$-$L$-bicomodule algebra. It remains to check the left Galois property. For $a \in A$, we have, in $H/H \varphi'(A^+)$
$$\overline{a_{(1)}\varphi(a_{(2)})} = \overline{a_{(1)}\varphi(a_{(2)}) a_{(3)}S(a_{(4)})}=
\overline{a_{(1)}\varepsilon(a_{(2)})S(a_{(3)})}=\varepsilon(a)\overline{1}$$
Then similarly to the proof in Lemma \ref{traright}, we have, for $a \in A$, 
$$\overline{a_{(1)}} \otimes \overline{S(a_{(2)})}=\varepsilon(a) \overline{1} \otimes \overline{1}$$
and we get a map
\begin{align*}
L &\longrightarrow H/H\varphi'(A^+) \otimes H/H\varphi'(A^+) \\
\pi(h) &\longmapsto \overline{h_{(1)}} \otimes \overline{S(h_{(2)})}
\end{align*}
that enables us to construct an inverse for the canonical map in a standard manner.
\end{proof}

\begin{remark}
{\rm That the right $L$-Galois object  $H / H\varphi'(A^+)$ is an $L$-bi-Galois object can also be deduced from the proof of Theorem 2 in \cite{mas}.}
\end{remark}

\begin{corollary}\label{corobigal}
 Let \begin{equation*}k\to A \overset{i}\to H \overset{\pi}\to L \to k
\end{equation*} be an exact sequence of Hopf algebras with $H$ faithfully flat as a right $A$-module. There exists a group morphism 
$$\mathcal T : \Alg_H(A,k) \longrightarrow {\rm BiGal}(L)$$
whose kernel consist of elements $\varphi \in \Alg_H(A,k)$ that have an extension $\tilde{\varphi} \in \Alg(H,k)$ and satisfy $\tilde{\varphi}(h_{(1)}) \pi(h_{(2)})= \tilde{\varphi}(h_{(2)}) \pi(h_{(1)})$ for any $h \in H$.

If moreover  ${\rm Cleft}(L)= {\rm Gal}(L)$ and ${\rm Cleft}(H)$ is trivial, then any $L$-bi-Galois object is isomorphic to ${^f \!\mathcal T}(\varphi)$ for some $\varphi \in \Alg_H(A,k)$ and $f \in {\rm Aut}_{\rm Hopf}(L)$.
\end{corollary}

\begin{proof}
 For $\varphi \in \Alg_H(A,k)$, denote by $T_\varphi$ the $L$-bi-Galois objects $H/H\varphi'(A^+)$ from Lemmas \ref{traright} and \ref{traleft}. For $\varphi, \psi \in \Alg_H(A,k)$, We have a map
\begin{align*}
 H  & \longrightarrow T_{\varphi} \square_L T_{\psi}\\
h & \longmapsto \overline{h_{(1)}} \otimes \overline{h_{(2)}}
\end{align*}
If $a \in A$, this maps send $\psi \varphi(a_{(1)})a_{(2)}= \psi(a_{(1)})\varphi(a_{(2)})a_{(3)}$ to
$$\psi(a_{(1)})\varphi(a_{(2)})\overline{a_{(3)}}\otimes \overline{a_{(4)}}= \psi(a_{(1)})\varepsilon(a_{(2)})\overline{1}\otimes \overline{a_{(3)}}= \varepsilon(a) \overline{1}\otimes \overline{1}$$
and therefore induces an algebra map 
$$ T_{\psi \varphi} \longrightarrow T_{\varphi} \square_L T_{\psi}$$
This algebra map is $L$-bicolinear, so is an isomorphism of $L$-bi-Galois objects. We get the announced group morphism 
\begin{align*}
 \Alg_H(A,k) &\longrightarrow {\rm BiGal}(L) \\ \varphi  & \longmapsto [T_{\varphi^{-1}}] 
\end{align*}
If $\varphi\in \Alg_H(A,k)$, then $\varphi$ is in the kernel if and only if there exists an algebra map $\psi : H/H (\varphi^{-1})'(A^+)\rightarrow k$ that
$\psi(\overline{h_{(1)}}) \pi(h_{(2)})=\psi(\overline{h_{(2)}}) \pi(h_{(1)})$ for any $h \in H$, if and only if there exists an algebra map $\psi_0 : H \rightarrow k$ such that $\varphi^{-1}(a_{(1)})\psi_0(a_{(2)})=\varepsilon(a)$ and 
$\psi_0(h_{(1)}) \pi(h_{(2)})=\psi_0(h_{(2)}) \pi(h_{(1)})$ for any $a \in A$ and $h \in H$. The first condition is equivalent to say that $\psi_0$ extends $\varphi$, and therefore the assertion about the kernel is proved.

If ${\rm Cleft}(L)= {\rm Gal}(L)$ and ${\rm Cleft}(H)$ is trivial, the map ${\rm Gal}(L) \rightarrow {\rm Gal}(H)$ induced by $\pi$ is then trivial, and hence if $T$ is a left $L$-bi-Galois object, Theorem \ref{schexact} ensures that
there exists $\varphi \in {\rm Alg}_H(A,k)$ such that $T \simeq T_\varphi=H/H\varphi'(A^+)$ as right $L$-Galois object. Then there exists $f$, a Hopf algebra automorphism of $L$ such that $T\simeq {^f\! T}_\varphi$ as $L$-bi-Galois objects.  This concludes the proof.
\end{proof}

\begin{remark}
{\rm The last conclusion in the previous result holds as soon as the map ${\rm Gal}(L) \rightarrow {\rm Gal}(H)$ induced by $\pi$ is then trivial.}
\end{remark}

\section{Application to $u_q(sl(n))^*$}

For $q \in k^*$, we denote by $x_{ij}$, $1\leq i,j\leq n$, the standard generators of $\mathcal O({\rm SL}_q(n))$, the coordinate algebra of the quantum group ${\rm SL}_q(n)$. We assume that $q$ is a root of unity of order $N>1$. Then there exists an injective Hopf algebra map 
\begin{align*}
 i : \mathcal O({\rm SL}(n)) &\longrightarrow \mathcal O({\rm SL}_q(n)) \\
x_{ij} &\longmapsto x_{ij}^N
\end{align*}
whose image is central. The corresponding Hopf algebra quotient 
$$\mathcal O({\rm SL}_q(n)_{1})= \mathcal O({\rm SL}_q(n))/\mathcal O({\rm SL}_q(n))\mathcal O({\rm SL}(n))^+$$
is then the quotient of $\mathcal O({\rm SL}_q(n))$ by the ideal generated by the elements $x_{ij}^N-\delta_{ij}$, $1\leq i,j \leq n$, see \cite{pw, tak}, and is known as the coordinate algebra of  the Frobenius-Lusztig Kernel of ${\rm SL}(n)$. It is known that $\mathcal O({\rm SL}_q(n)_{1})\simeq u_q(sl(n))^*$ (see \cite{tak}), and we freely write 
this isomorphism as an equality. 

We thus have a central sequence of Hopf algebras
\begin{equation*}k \to  \mathcal O({\rm SL}(n)) \overset{i}\to \mathcal O({\rm SL}_q(n)) \overset{\pi}\to   u_q(sl(n))^* \to k\end{equation*}
with $\mathcal O({\rm SL}_q(n))$ faithfully flat (since central) as an $\mathcal O({\rm SL}(n))$-module (in fact it is a free module, see e.g. III.7.11 in \cite{bg}), and hence this is an exact sequence.
Thus we can use the considerations in the previous section. The centrality of $ \mathcal O({\rm SL}(n))$ ensures that
 $\Alg_{\mathcal O({\rm SL}_q(n))}( \mathcal O({\rm SL}(n)),k)= \Alg( \mathcal O({\rm SL}(n)),k)\simeq {\rm SL}(n)$.
Therefore Corollary \ref{corobigal} yields a group morphism
\begin{align*}
\mathcal T : {\rm SL}(n) & \longrightarrow {\rm BiGal}(u_q(sl(n))^*)\\
g & \longmapsto [T_g]
\end{align*}
For $g=(g_{ij}) \in {\rm SL}(n)$, it is straightforward to see that the algebra $T_g$ is the quotient of  $\mathcal O({\rm SL}_q(n))$ by the ideal generated by the elements $x_{ij}^N-g_{ij}$, and that its $u_q(sl(n))^*$-bicomodule algebra structure is given by
\begin{align*}
 \rho : T_g &\longrightarrow T_g \otimes u_q(sl(n))^*    \quad &\lambda :  T_g &\longrightarrow u_q(sl(n))^* \otimes T_g \\
x_{ij}& \longmapsto \sum_k x_{ik}\otimes x_{kj} \quad     &  x_{ij} &\longmapsto \sum_k x_{ik}\otimes x_{kj}
\end{align*}
where, with a slight abuse of notation, we still denote $x_{ij}$ the generators in all the considered quotients of $\mathcal O({\rm SL}_q(n))$.
An element $g \in {\rm SL}(n)$ has an extension to an algebra map on   $\mathcal O({\rm SL}_q(n))$ if and only if it is diagonal, and will satisfy the other condition defining the kernel of our map in Corollary \ref{corobigal} if and only it is a scalar matrix (the elements $x_{ij}$ being linearly independent in $u_q(sl(n))^*$). Therefore we get the embedding ${\rm PSL}(n) \hookrightarrow {\rm BiGal}(u_q(sl(n))^*)$ of Theorem \ref{main}.

To prove the assertion at $n=2$ in Theorem \ref{main}, we need some more considerations. Recall that for $r=(r_1, \ldots , r_n)\in (k^*)^n$ with $r_1 \cdots r_n=1$, we have a Hopf algebra automorphism of $\mathcal O({\rm SL}_q(n))$ defined by $x_{ij} \mapsto r_{i}^{-1}r_{j}x_{ij}$. It is straightforward to see that this automorphism induces an automorphism of $u_q(sl(n))^*$, denoted $f_r$.

\begin{lemma}\label{trans}
 For any $g$ in ${\rm SL}(n)$ and $r=(r_1, \ldots , r_n)\in (k^*)^n$ with $r_1\cdots r_n=1$, there exists $g'\in {\rm SL}(n)$ such that ${^{f_{r}}\! T}_g\simeq T_{g'}$ as   $u_q(sl(n))^*$-bi-Galois objects.
\end{lemma}

\begin{proof}
 Define $g'$ by $g'_{ij}= r_i^Ng_{ij}$. It is a straightforward verification to check that there exists an algebra map
\begin{align*}
 T_{g'} & \longrightarrow{^{f_{r}}\! T}_{g} \\
x_{ij}  & \longmapsto r_ix_{ij}
\end{align*}
which is bicolinear, hence an isomorphism.
\end{proof}

Assume now that $n=2$. For $r\in k^*$, we denote again by $r$ the element $(r,r^{-1})$ in $(k^*)^2$, and $f_r$ the corresponding automorphism of $u_q(sl(2))^*$. The following lemma is probably well-known.

\begin{lemma}
 Any Hopf algebra automorphism of $u_q(sl(2))^*$ is of the form $f_r$ for some $r \in k^*$, and is co-inner if and only if $r^N=1$.
\end{lemma}

\begin{proof}
Let $f$ be a Hopf algebra automorphism of $u_q(sl(2))^*$.
It is known (see e.g. \cite{pw}, Chapter 9, or \cite{kas}, Chapter VI) that there are exactly $N$ isomorphism classes of simple   $u_q(sl(2))^*$-comodules, of respective dimensions $1, \ldots ,N$. Hence the tensor equivalence on comodules induced by $f$ must preserve each such comodule. The two-dimensional one has $(x_{ij})$
as matrix of coefficients, hence there exists an invertible matrix $P$ such that
$$f(\begin{pmatrix} x_{11} & x_{12 } \\ x_{21} & x_{22}\end{pmatrix}) = P \begin{pmatrix} x_{11} & x_{12 } \\ x_{21} & x_{22}\end{pmatrix}P^{-1}$$
Arguing as in the proof of Theorem 5.3 in \cite{bi1}, we see that $P$ is diagonal, and then that $f=f_r$ for some $r\in k^*$. The last assertion is immediate.
\end{proof}

We have ${\rm Gal}(u_q(sl(2))^*)={\rm Cleft}(u_q(sl(2))^*)$ since this Hopf algebra is finite-dimensional. Moreover ${\rm Cleft}(\mathcal O({\rm SL}_q(2))$ is trivial if $q \not=1$, see the first part of the proof of Corollary 16 in \cite{au}. Hence, by Corollary \ref{corobigal}, any $u_q(sl(2))^*$-bi-Galois object is isomorphic to ${^f\!T}_g$ for some $g \in {\rm SL}(2)$ and a Hopf algebra automorphism $f$ of $u_q(sl(2))^*$.
The combination of the previous two  lemmas then shows the surjectivity of the map $\mathcal T$, and this concludes the proof of Theorem \ref{main}.

\begin{remark}
{\rm  It follows also from the fact that the map ${\rm Gal}(u_q(sl(2)))^*) \rightarrow {\rm Gal}(\mathcal O({\rm SL}_q(2))$ is trivial  and Schauenburg's exact sequence that there is a bijection ${\rm SL}(2)/k^* \simeq {\rm Gal}(u_q(sl(2))^*)$, where $k^*$ is viewed as the diagonal subgroup of ${\rm SL}(2)$. Also we see that any right $u_q(sl(2))^*$-Galois object is in fact a $u_q(sl_2))^*$-bi-Galois object: this implies (\cite{sc1}) that $u_q(sl(2))^*$ is categorically rigid in the sense that if $H$ is any Hopf algebra having its tensor category of comodules equivalent to the one of $u_q(sl(2))^*$, then $H$  is isomorphic to $u_q(sl(2))^*$. Such a property was shown for the Taft algebras in \cite{sc00}, but is not true for $u_q(sl(2))$ (\cite{scha}).}
\end{remark}

\begin{remark}
{\rm An action of ${\rm SL}(n)$ on the category of $u_q(sl(n))$-modules, in the even root of unity case, was also constructed in \cite{ag}, Corollary 2.5.}
\end{remark}

We now discuss the lazy cohomology of $ u_q(sl(2))^*$. We will use the following easy lemma.

\begin{lemma}
 Let $q : L \rightarrow B$ be a surjective Hopf algebra map. Assume that there exists a cocycle $\sigma$ on $B$ that is not lazy. Then the cocycle $\sigma_q= \sigma (q \otimes q)$ on $L$ is not lazy. 
\end{lemma}

\begin{proof}
 Let $a,b \in B$ be such that $\sigma(a_{(1)},b_{(1)})a_{(2)}b_{(2)}\not= \sigma(a_{(2)},b_{(2)})a_{(1)}b_{(1)}$.
For $x,y \in L$ with $q(x)=a$ and $q(y)=b$, we have $q(\sigma_q(x_{(1)},y_{(1)})x_{(2)}y_{(2)}- \sigma_q(x_{(2)},y_{(2)})x_{(1)}y_{(1)}) \not=0$, which gives the result.
\end{proof}

\begin{corollary}
 The lazy cohomology group ${\rm H}^2_{\ell}(u_q(sl(2))^*)$ is trivial.
\end{corollary}

\begin{proof}
 Since ${\rm PSL}(2)$ is a simple group and ${\rm H}^2_{\ell}(u_q(sl(2))^*)$ is a normal subgroup, it is enough to check that there exists a bi-Galois object that is not bicleft, or equivalently that there exists a cocycle on $u_q(sl(2))^*$ that is not lazy. Consider $B$, the quotient of $u_q(sl(2))^*$ by the relation $x_{21}=0$. This is a Hopf algebra, isomorphic to a Taft algebra. It is known that the Taft algebras have cocycles that are not lazy \cite{bc}, and hence
the previous lemma concludes the proof.
\end{proof}

\bibliographystyle{amsalpha}

\end{document}